\documentclass[12pt]{amsart}
\usepackage{amsmath,amsthm,amsfonts,amssymb,latexsym}

\usepackage{hyperref}

\usepackage{enumerate}
\usepackage[shortlabels]{enumitem}

\begin{document}

\theoremstyle{plain}

\newtheorem{thm}{Theorem}[section]
\newtheorem{lem}[thm]{Lemma}
\newtheorem{conj}[thm]{Conjecture}
\newtheorem{pro}[thm]{Proposition}
\newtheorem{cor}[thm]{Corollary}
\newtheorem{que}[thm]{Question}
\newtheorem{rem}[thm]{Remark}
\newtheorem{defi}[thm]{Definition}

\newtheorem*{thmA}{THEOREM A}
\newtheorem*{thmB}{THEOREM B}
\newtheorem*{conjC}{CONJECTURE C}
\newtheorem*{conjD}{CONJECTURE D}

\newtheorem*{thmAcl}{Main Theorem$^{*}$}
\newtheorem*{thmBcl}{Theorem B$^{*}$}

\numberwithin{equation}{section}

\newcommand{\Maxn}{\operatorname{Max_{\textbf{N}}}}
\newcommand{\Syl}{\operatorname{Syl}}
\newcommand{\dl}{\operatorname{\mathfrak{d}}}
\newcommand{\Con}{\operatorname{Con}}
\newcommand{\cl}{\operatorname{cl}}
\newcommand{\Stab}{\operatorname{Stab}}
\newcommand{\Aut}{\operatorname{Aut}}
\newcommand{\Ker}{\operatorname{Ker}}
\newcommand{\IBr}{\operatorname{IBr}}
\newcommand{\Irr}{\operatorname{Irr}}
\newcommand{\SL}{\operatorname{SL}}
\newcommand{\FF}{\mathbb{F}}
\newcommand{\NN}{\mathbb{N}}
\newcommand{\N}{\mathbf{N}}
\newcommand{\C}{\mathbf{C}}
\newcommand{\OO}{\mathbf{O}}
\newcommand{\F}{\mathbf{F}}

\newcommand{\bG}{\mathbf{G}}
\newcommand{\bL}{\mathbf{L}}
\newcommand{\bN}{\mathbf{N}}
\newcommand{\uch}{{\operatorname{UCh}}}

\renewcommand{\labelenumi}{\upshape (\roman{enumi})}

\newcommand{\GL}{\operatorname{GL}}
\newcommand{\Sp}{\operatorname{Sp}}
\newcommand{\PGL}{\operatorname{PGL}}
\newcommand{\PSL}{\operatorname{PSL}}
\newcommand{\SU}{\operatorname{SU}}
\newcommand{\PSU}{\operatorname{PSU}}
\newcommand{\PSp}{\operatorname{PSp}}

\providecommand{\V}{\mathrm{V}}
\providecommand{\E}{\mathrm{E}}
\providecommand{\ir}{\mathrm{Irr_{rv}}}
\providecommand{\Irrr}{\mathrm{Irr_{rv}}}
\providecommand{\re}{\mathrm{Re}}

\def\irrp#1{{\rm Irr}_{p'}(#1)}

\def\Z{{\mathbb Z}}
\def\C{{\mathbb C}}
\def\Q{{\mathbb Q}}
\def\irr#1{{\rm Irr}(#1)}
\def\ibr#1{{\rm IBr}(#1)}
\def\irrv#1{{\rm Irr}_{\rm rv}(#1)}
\def \c#1{{\cal #1}}
\def\cent#1#2{{\bf C}_{#1}(#2)}
\def\syl#1#2{{\rm Syl}_#1(#2)}
\def\nor{\trianglelefteq\,}
\def\oh#1#2{{\bf O}_{#1}(#2)}
\def\Oh#1#2{{\bf O}^{#1}(#2)}
\def\zent#1{{\bf Z}(#1)}
\def\det#1{{\rm det}(#1)}
\def\ker#1{{\rm ker}(#1)}
\def\norm#1#2{{\bf N}_{#1}(#2)}
\def\alt#1{{\rm Alt}(#1)}
\def\iitem#1{\goodbreak\par\noindent{\bf #1}}
   \def \mod#1{\, {\rm mod} \, #1 \, }
\def\sbs{\subseteq}

\def\gc{{\bf GC}}
\def\ngc{{non-{\bf GC}}}
\def\ngcs{{non-{\bf GC}$^*$}}
\newcommand{\notd}{{\!\not{|}}}
\newcommand{\Out}{{\mathrm {Out}}}
\newcommand{\Mult}{{\mathrm {Mult}}}
\newcommand{\Inn}{{\mathrm {Inn}}}
\newcommand{\IBR}{{\mathrm {IBr}}}
\newcommand{\IBRL}{{\mathrm {IBr}}_{\ell}}
\newcommand{\IBRP}{{\mathrm {IBr}}_{p}}
\newcommand{\ord}{{\mathrm {ord}}}
\def\id{\mathop{\mathrm{ id}}\nolimits}
\renewcommand{\Im}{{\mathrm {Im}}}
\newcommand{\Ind}{{\mathrm {Ind}}}
\newcommand{\diag}{{\mathrm {diag}}}
\newcommand{\soc}{{\mathrm {soc}}}
\newcommand{\End}{{\mathrm {End}}}
\newcommand{\sol}{{\mathrm {sol}}}
\newcommand{\Hom}{{\mathrm {Hom}}}
\newcommand{\Mor}{{\mathrm {Mor}}}
\newcommand{\St}{{\sf {St}}}
\def\rank{\mathop{\mathrm{ rank}}\nolimits}
\newcommand{\Tr}{{\mathrm {Tr}}}
\newcommand{\tr}{{\mathrm {tr}}}
\newcommand{\Gal}{{\it Gal}}
\newcommand{\Spec}{{\mathrm {Spec}}}
\newcommand{\ad}{{\mathrm {ad}}}
\newcommand{\Sym}{{\mathrm {Sym}}}
\newcommand{\Char}{{\mathrm {char}}}
\newcommand{\pr}{{\mathrm {pr}}}
\newcommand{\rad}{{\mathrm {rad}}}
\newcommand{\abel}{{\mathrm {abel}}}
\newcommand{\codim}{{\mathrm {codim}}}
\newcommand{\ind}{{\mathrm {ind}}}
\newcommand{\Res}{{\mathrm {Res}}}
\newcommand{\Ann}{{\mathrm {Ann}}}
\newcommand{\Ext}{{\mathrm {Ext}}}
\newcommand{\Alt}{{\mathrm {Alt}}}
\newcommand{\AAA}{{\sf A}}
\newcommand{\SSS}{{\sf S}}
\newcommand{\CC}{{\mathbb C}}
\newcommand{\CB}{{\mathbf C}}
\newcommand{\RR}{{\mathbb R}}
\newcommand{\QQ}{{\mathbb Q}}
\newcommand{\ZZ}{{\mathbb Z}}
\newcommand{\KK}{{\mathbb K}}
\newcommand{\NB}{{\mathbf N}}
\newcommand{\ZB}{{\mathbf Z}}
\newcommand{\OB}{{\mathbf O}}
\newcommand{\EE}{{\mathbb E}}
\newcommand{\PP}{{\mathbb P}}
\newcommand{\GC}{{\mathcal G}}
\newcommand{\HC}{{\mathcal H}}
\newcommand{\AC}{{\mathcal A}}
\newcommand{\BC}{{\mathcal B}}
\newcommand{\GA}{{\mathfrak G}}
\newcommand{\SC}{{\mathcal S}}
\newcommand{\TC}{{\mathcal T}}
\newcommand{\DC}{{\mathcal D}}
\newcommand{\LC}{{\mathcal L}}
\newcommand{\RC}{{\mathcal R}}
\newcommand{\CL}{{\mathcal C}}
\newcommand{\EC}{{\mathcal E}}
\newcommand{\GCD}{\GC^{*}}
\newcommand{\TCD}{\TC^{*}}
\newcommand{\FD}{F^{*}}
\newcommand{\GD}{G^{*}}
\newcommand{\HD}{H^{*}}
\newcommand{\hG}{\hat{G}}
\newcommand{\hP}{\hat{P}}
\newcommand{\hQ}{\hat{Q}}
\newcommand{\hR}{\hat{R}}
\newcommand{\GCF}{\GC^{F}}
\newcommand{\TCF}{\TC^{F}}
\newcommand{\PCF}{\PC^{F}}
\newcommand{\GCDF}{(\GC^{*})^{F^{*}}}
\newcommand{\RGTT}{R^{\GC}_{\TC}(\theta)}
\newcommand{\RGTA}{R^{\GC}_{\TC}(1)}
\newcommand{\Om}{\Omega}
\newcommand{\eps}{\epsilon}
\newcommand{\varep}{\varepsilon}
\newcommand{\al}{\alpha}
\newcommand{\chis}{\chi_{s}}
\newcommand{\sigmad}{\sigma^{*}}
\newcommand{\PA}{\boldsymbol{\alpha}}
\newcommand{\gam}{\gamma}
\newcommand{\lam}{\lambda}
\newcommand{\la}{\langle}
\newcommand{\ra}{\rangle}
\newcommand{\hs}{\hat{s}}
\newcommand{\htt}{\hat{t}}
\newcommand{\sgn}{\mathsf{sgn}}
\newcommand{\SR}{^*R}
\newcommand{\bT}{\mathbf{T}}
\newcommand{\tn}{\hspace{0.5mm}^{t}\hspace*{-0.2mm}}
\newcommand{\ta}{\hspace{0.5mm}^{2}\hspace*{-0.2mm}}
\newcommand{\tb}{\hspace{0.5mm}^{3}\hspace*{-0.2mm}}
\def\skipa{\vspace{-1.5mm} & \vspace{-1.5mm} & \vspace{-1.5mm}\\}
\newcommand{\tw}[1]{{}^#1\!}
\renewcommand{\mod}{\bmod \,}
\newcommand{\edit}[1]{{\color{red} #1}}
\renewcommand{\emptyset}{\varnothing}

\newcommand\type[1]{\operatorname{#1}}

\marginparsep-0.5cm

\renewcommand{\thefootnote}{\fnsymbol{footnote}}
\footnotesep6.5pt

\title{The Eaton--Moret\'o Conjecture and  $p$-solvable groups}

\author[Gabriel Navarro]{Gabriel Navarro}
\address[Gabriel Navarro]{Departament de Matem\`atiques, Universitat de Val\`encia, 46100 Burjassot,
Val\`encia, Spain}
\email{gabriel@uv.es}
  
\thanks{This research   is supported by Grant  PID2022-137612NB-I00 funded by MCIN/AEI/ 10.13039/501100011033 and ERDF ``A way of making Europe." The author would like to thank Alex Moret\'o,
J. M. Mart{\'\i}nez and Noelia Rizo for useful conversations on the subject.}



\begin{abstract}
We prove that the Eaton--Moret\'o conjecture is true for the principal blocks of the $p$-solvable groups.
\end{abstract}

\maketitle

\section{Introduction}   
The Eaton--Moret\'o conjecture (\cite{EM}) is an extension of the famous Brauer's Height Zero conjecture (recently proved in \cite{MNST} and \cite{Ru},
after   \cite{KM}). It proposes an exciting equality between the irreducible character degrees in Brauer $p$-blocks and those of its defect groups.  Suppose that $B$ is a $p$-block of a finite group $G$ with defect group $D$. Write $|G|_p=p^a$ and $|D|=p^d$, where $n_p$ is the largest $p$-power dividing the integer $n$. Richard Brauer proved that
the minimum of the $p$-parts $\chi(1)_p$ of the degrees $\chi(1)$ of the irreducible
characters $\chi \in \irr B$ in $B$ is $p^{a-d}$. In particular, if $\chi \in \irr B$ then $\chi(1)_p=p^{a-d +h}$ for a unique integer $h=h(\chi) \ge0$,   called the
height of $\chi$.  Brauer's Height Zero Conjecture, now a theorem, asserts that $D$ is abelian if and only if $h(\chi)=0$ for all $\chi \in \irr B$. 
Suppose now that $D$ is not abelian, and therefore that $h(\chi)>0$ for some $\chi \in \irr B$. If $mh(B)$ is the minimum of the non-zero heights of $\irr B$, then the Eaton--Moret\'o conjecture proposes that
$$mh(B)=mh(D) \, .$$
This conjecture has been extensively studied (\cite{EM}, \cite{BM}, \cite{FLZ}, \cite{MMR}, etc)
and it has strong support,  although mainly outside $p$-solvable groups. 
(There is even a version of this conjecture for fusion systems in \cite{KLLS}.)
It is therefore   for $p$-solvable groups where evidence for the conjecture is  weaker. Until now.

   \begin{thmA}
  Suppose that $B$ is the principal $p$-block of a  finite $p$-solvable group $G$ with a non-abelian defect group $D$.
  Then $mh(B)=mh(D)$.
   \end{thmA}
   
 Once the principal block of the $p$-solvable case of the conjecture is established, it is now natural to ask if  a reduction of the conjecture to a question on simple groups
   is possible, at least in the principal block case. As shown in \cite{MMR},   
   this appears to be difficult.   It  also seems difficult, if not impossible, to prove the general $p$-solvable case of the conjecture with the techniques used in this paper. Some new idea is needed.
   
     \medskip
   
   We do remind the reader that the inequality $mh(D) \le mh(B)$ in the conjecture is a consequence of the well established Dade--Robinson Conjecture (\cite{Ro});
   so it is on the inequality $mh(B) \le mh(D)$ where the main focus in this note is.
   \medskip
   
   Finally, we remark that Theorem A solves  Conjecture 4.4 in \cite{N3}.

   \section{Preliminaries}
   Our notation for ordinary characters follows \cite{Is} and \cite{N2}.
   Our notation for blocks follows \cite{N1}.
   In this section, we collect the key results used in the proof of Theorem A for the reader's convenience.
   
   \begin{lem}\label{tom}
   Assume that $A$ acts on $G$ via automorphisms,
   $N\nor G$ is $A$-invariant, $(|G:N|,|A|)=1$,  and $\cent{G/N}A=G/N$.
   Let $\chi \in \irr G$ and $\theta \in \irr N$ such that $[\chi_N, \theta]\ne 0$.
   Then $\chi$ is $A$-invariant if and only if $\theta$ is $A$-invariant.
   \end{lem}
   
   \begin{proof}
   This Lemma 1.4 of \cite{W}.
    \end{proof}
    
    \begin{lem}\label{easy}
    Suppose that $G/N$ is a $p$-group, and let $H\le G$ such that $G=NH$. Write $M=N\cap H$. Let $\theta \in \irr N$ be $G$-invariant, and let
    $\varphi \in \irr M$ be $H$-invariant such that $[\theta_M,\varphi]$ is not divisible by $p$. Suppose that $\theta$ extends to
    $\chi \in \irr G$. Then $\varphi$ extends to $H$.
    \end{lem}
    
    \begin{proof}
   Write $\chi_H=e_1 \mu_1 + \ldots +e_t \mu_t$, where $\mu_i \in \irr H$
   and $e_i$ are positive integers.
   Thus 
   $$[\chi_M, \varphi]= e_1 [(\mu_1)_M, \varphi] + \ldots +e_t [(\mu_t)_M, \varphi] \, .$$
   Hence, there is $i$ such that $[\mu_i, \varphi] $ is not divisible by $p$.
   Since $H/M$ is a $p$-group and $\varphi$ is $H$-invariant, we conclude that $(\mu_i)_M=\varphi$, by using Corollary 11.29 of \cite{Is}. 
     \end{proof}
     
     If a group $A$ acts by automorphisms on another group $G$, it is common to write ${\rm Irr}_A(G)$ for the set of the irreducible
     characters of $G$ that are $A$-invariant. The following is the {\sl relative} Glauberman correspondence.

    \begin{thm}\label{relg}
    Suppose that a $p$-group $P$ acts as  automorphisms on a finite group $G$. Let $N \nor G$ be $P$-invariant such that $G/N$ is a $p'$-group.
    Let $C/N=\cent {G/N}P$. Then there exists a natural bijection
    $^*: {\rm Irr}_P(G) \rightarrow {\rm Irr}_P(C)$. In fact, if $\chi \in {\rm Irr}_P(G)$ then
    $$\chi_C=e\chi^* +p \Delta + \Xi \, ,$$
    where $\Delta$ and $\Xi$ are characters of $C$ or zero, $p$ does not divide $e$, and no irreducible constituent of $\Xi$ lies over some $P$-invariant character of $N$. Furthermore, $[\chi_C, \chi^*]$ is not divisible by $p$.
    \end{thm}
    
    \begin{proof}
    This is Theorem E of \cite{NSV}. Using Lemma \ref{tom},
    notice that $\chi^*$ is not a constituent of $\Xi$. Hence 
    $[\chi_C, \chi^*] \equiv e$ mod $p$, is not divisible by $p$.
    \end{proof}

    \begin{lem}\label{act}
    
    Suppose that a $p$-group $P$ acts coprimely as  automorphisms on a finite group $G$.     Let  
      $Q \le P$. If ${\rm Irr}_P(G)={\rm Irr}_Q(G)$,
    then $\cent {G}P=\cent{G}Q$.
        \end{lem}
    
    \begin{proof}
    This follows from Lemma 2.2 of \cite{N0}.
    \end{proof}

    \section{Proof of Theorem A}
 
 If $G$ is a finite group, $N \nor G$ and $\theta \in \irr N$, recall that $\irr{G|\theta}$ is the set of irreducible characters of $G$ such that the restriction $\chi_N$ contains $\theta$.
 If $P$ is a non-abelian $p$-group, let $m(P)$ be the minimum
 character degree among the non-linear irreducible characters of $P$.

  \begin{thm}\label{main}
   Suppose that $G$ is a finite $p$-solvable group,    and $P \in \syl pG$ is not abelian. Assume that $\oh{p'}G=1$.  If $m(P)=p^a$,   then there
 is $\chi \in \irr G$ such that $1 \ne \chi(1)_p \le p^a$.
   \end{thm}
   \begin{proof}
   By induction on $|G|$.   Let $V=\Oh{p'} G$. 
   Hence $P \in \syl pV$. If $V<G$, by induction let 
    $\varphi \in \irr{V}$ be such that $1 \ne \varphi(1)_p \le p^a$.
     Hence if $\chi \in \irr{V|\varphi}$ then we have that $\chi(1)_p=\varphi(1)_p$  by Corollary 11.29 of \cite{Is}, and we are done. 
Hence   we may assume therefore that $G=\Oh{p'} G$.
   Let $K=\Oh pG$  and let $L=\Oh {p'}{K}$. If $K=L$, then $G$
   is a $p$-group, and  in this case, the theorem is clear.   
      Let $U=LP<G$.       By induction and working in $U/\oh{p'}U$, there exists $\tau \in \irr{U}$
   such that $1<\tau(1)_p\le p^a$. 
   Let $\lambda \in \irr L$ be under $\tau$, let $Q=U_\lambda$ be the stabilizer of $\lambda$ in $U$. 
   Write $$\tau^G=a_1\psi_1 + \ldots + a_t\psi_t\, ,$$ where $\psi_i \in \irr G$
   and the $a_i$ are positive integers.
   Since $\tau^G(1)_p=\tau(1)_p$, 
 we cannot have that $\psi_i(1)_p>p^a$ for all $i$. So there is $\psi=\psi_i$ such that $\psi(1)_p\le p^a$. If  $1 \ne \psi(1)_p$, then we are done. 
 So we may assume that $\psi$ has $p'$-degree. Thus $\psi_K \in \irr K$,
 using Corollary 11.29 of \cite{Is}. Also, $[\psi_L, \lambda]$ is not divisible by $p$ by the same result. 
 Notice that, by Clifford's theorem, we have that $\lambda$ has $p'$-degree because it lies under $\psi$.
 If $\lambda$ is $P$-invariant, then $\lambda$ extends to $P$
 (by Lemma \ref{easy}). Let $\tilde\lambda \in \irr U$ be extending $\lambda$. 
 By Gallagher's Corollary 6.19 of \cite{Is}, we have that   $\tau=\tilde\lambda \nu$, where $\tilde\lambda \in \irr U$ extends $\lambda$ and 
 $\nu \in \irr{U/L}$ with $\nu(1)_p=\tau(1)_p$. Since $U/L \cong G/K$, then
 we can see  $\nu \in \irr{G/K}$, and we are done in this case.  Therefore we may assume that $Q<U$. 
 Let $\mu \in \irr Q$ be the Clifford correspondent of $\tau$ over $\lambda$. 
 Then $|U:Q|\mu(1)=\tau(1)$ and therefore $|U:Q|\le p^a$.
 Let $R=Q\cap P$. Then $1<|P:R|=|U:Q|\le p^a$.
 If $\eta$ is an irreducible constituent of $(1_R)^P$, then $\eta(1)$ divides $|P:R|\le p^a$ (using that $R\nor\nor G$ and Corollary 11.29 of \cite{Is}). 
 Since $m(P)=p^a$, we conclude that $\eta$ is linear.  
 We conclude that $P' \sbs R$.
 Hence, $R\nor P$ and $Q=LR \nor U$. 
 Also $M=KQ \nor G$. Let $\gamma \in \irr M$ of degree not divisible by $p$. Let $\chi \in \irr{G|\gamma}$.
 Then $\chi(1)/\gamma(1)$ divides $|G:M|=|P:R|\le p^a$, and therefore
 $\chi(1)_p\le p^a$. Then $\chi$ has necessarily $p'$-degree  and thus $\chi_M=\gamma$. 
 
 We claim that $\cent{K/L}{U/L}=\cent{K/L}{Q/L}$. 
 We have that $U$ acts on $K/L$ by conjugation
 $(kL)^u=k^uL$, with $L$ in the kernel of the action.
 Hence $U/L$ acts coprimely on $K/L$. By using Lemma \ref{act}, we only need to show that every $Q$-invariant character of $K/L$ is $U$-invariant.  If $\eta \in \irr{K/L}$ is $Q$-invariant, then
 $\eta$ extends to $\tilde \eta \in \irr{KQ}=\irr{M}$ by Corollary 6.28 of \cite{Is}, for instance. Now $\tilde\eta$ has degree not divisible by $p$,
 so we have that  $\tilde\eta$ extends to $G$, by the previous paragraph. 
 This proves the claim.
 If $C/L=\cent{K/L}{U/L}=\cent{K/L}{Q/L}$, notice too that
 $\cent{K/L}P=\cent{K/L} R=C$.  Further notice that $N=\norm GQ=\norm GU$. Indeed, since $M\cap U=Q$ and $M\nor G$, we have that $\norm GU \sbs \norm GQ$. Since $\norm KQ/L=\cent{K/L}{Q/L}$ and
 $\norm KU/L=\cent{K/L}{U/L}$, we conclude that $\norm GQ=\norm GU$.

  Let $\rho \in \irr{N}$ be over $\tau$, and  
 let $\epsilon \in \irr{C}$ be under $\rho$ and over $\lambda$. 
Since $\lambda$ is $Q$-invariant, we have that $\lambda$ is $R$-invariant.   By Lemma \ref{tom},
 we have that $\epsilon$ is $R$-invariant, and therefore we have that $\epsilon$ is $Q$-invariant. By Theorem \ref{relg},
 we have $\epsilon=\xi^*$ for some $\xi \in {\rm Irr}_Q(K)$.
 Also, 
 we have that $\xi_C=e\xi^* + p \Delta + \Xi$, where $p$ does not divide $e$, $\Delta$ and $\Xi$ are characters of $C$ or zero, and no irreducible constituent of $\Xi$ lies over some $Q$-invariant character of $L$.
 Hence, no irreducible constituent of $\Xi$ lies over some $P$-invariant character of $L$. Also $[\xi_C, \xi^*]$ is not divisible by $p$ (in particular, is not zero), and thus $\xi$ lies over $\lambda$ . In particular, since $|K/L|$ has order not divisible by $p$, we have that $\xi$ has degree not divisible by $p$.
 Since $\Oh pK=K$, we have that the determinantal order of $\xi$ is not divisible by $p$. Hence, we conclude that $\xi$ extends to some $\tilde\xi \in \irr M$ by Corollary 6.28 of \cite{Is}. Now $\tilde\xi$ has $p'$-degree, and therefore we know that $\tilde\xi$ extends to $G$. We conclude that $\xi$ is $P$-invariant.
 Again by Theorem \ref{relg},
 we also  have that
 $\xi_C=e_1\hat\xi + p \Delta_1 + \Xi_1$, where $\Delta_1$ and $\Xi_1$ are characters of $C$ or zero, $e_1$ is not divisible by $p$, $[\xi_C, \hat\xi]$ is not divisible by $p$, and no irreducible constituent of
 $\Xi_1$ lies over some $P$-invariant character of $L$. 
  We have that $[\Xi, \hat\xi]=0$, since any irreducible constituent of $\hat\xi_L$ would be $P$-invariant by Lemma \ref{tom}.
  Hence
  $$[\xi_C, \hat\xi]=e[\xi^*, \hat\xi] + p [\Delta, \hat\xi] \not\equiv \, 0 \, {\rm mod} \, p \, .$$ Necessarily, we conclude that $\hat\xi=\xi^*$ is $P$-invariant.
  Now, $\lambda$ lies under $\xi^*=\epsilon$. By Lemma \ref{tom}, 
 we conclude that $\lambda$ is $P$-invariant, which is a contradiction,
 since we already established that $Q<P$.
 \end{proof}

 \medskip
 The following is Theorem A of the Introduction.
 
 \begin{cor}
 Suppose that $G$ is $p$-solvable and $B$ is the principal $p$-block of $G$ with
 non-abelian defect group $P$. Then $mh(B)=mh(P)$.
 \end{cor}
 
 \begin{proof}
 Since $B$ is the principal block,
 we have that $P \in \syl pG$.
 By Theorem B of \cite{EM}, we only need to show that $mh(B)\le mh(P)$.
 Suppose that $1<p^a$ is the smallest of the irreducible character degrees of $P$. We need to show that there is $\chi \in \irr B$ such that $1 \ne \chi(1)_p \le p^a$.  By Theorem 10.20 of \cite{N1}, we have that
  $\irr B=\irr{G/\oh{p'}G}$.  Hence, we want to show that there exists $\chi \in \irr{G/\oh{p'}G}$   such that
   $1 \ne \chi(1)_p \le p^a$. But this follows from Theorem \ref{main}.
 \end{proof}

 There does not seem to exist a direct relationship between the irreducible characters $\alpha$ of $P\in \syl pG$ with
 $\alpha(1)=m(P)$ and the irreducible characters of the principal block $B_0$ of $G$ whose degree has $p$-part  $m(P)$, or at least one that  respects  fields of values (say, over the $p$-adics). 
 For instance, in ${\tt SmallGroup}(96,64)$,  for $p=2$, we have that $m(P)=2$, and the irreducible characters of
 degree 2 of $P$ have field of values $\Q$ and $\Q(i)$; 
 on the other hand, the field of values of the irreducible characters in the principal $2$-block of $G$ whose 2-part is $2$ are rational-valued.
 Continuing with $p=2$, say, and writing   
 $c(\chi)$ for the conductor of the character $\chi$ (that is, $c(\chi)$ is the smallest positive integer $n$ such that
 the values of $\chi$ are in the cyclotomic field $\Q_n$), it would be interesting to compare ${\rm min}\{(c(\alpha)\,  |\, \alpha \in \irr P, \alpha(1)=m(P)\}$ with ${\rm min}\{(c(\chi)_p\,  |\, \chi \in \irr{B_0},  \, \chi(1)_p=m(P)\}$ for any finite group. We will refrain from making any formal statement.

\end{document}